\documentclass[11pt]{article}

\usepackage[margin=1in]{geometry}
\usepackage[T1]{fontenc}
\usepackage{lmodern}
\usepackage{amsmath,amssymb,amsthm,mathtools,bm,array}
\usepackage{enumitem}
\usepackage{microtype}
\usepackage[round,authoryear]{natbib}
\usepackage[colorlinks=true,linkcolor=blue,citecolor=blue,urlcolor=blue]{hyperref}
\usepackage[capitalise,noabbrev]{cleveref}

\newtheorem{theorem}{Theorem}[section]
\newtheorem{proposition}[theorem]{Proposition}
\newtheorem{lemma}[theorem]{Lemma}
\newtheorem{corollary}[theorem]{Corollary}

\crefname{theorem}{Theorem}{Theorems}
\Crefname{theorem}{Theorem}{Theorems}
\crefname{proposition}{Proposition}{Propositions}
\Crefname{proposition}{Proposition}{Propositions}
\crefname{lemma}{Lemma}{Lemmas}
\Crefname{lemma}{Lemma}{Lemmas}
\crefname{corollary}{Corollary}{Corollaries}
\Crefname{corollary}{Corollary}{Corollaries}
\crefname{appendix}{Appendix}{Appendices}
\Crefname{appendix}{Appendix}{Appendices}

\newcommand{\mP}{\mathbb P}
\newcommand{\mE}{\mathbb E}
\newcommand{\R}{\mathbb R}
\newcommand{\N}{\mathbb N}
\newcommand{\calA}{\mathcal A}
\newcommand{\calE}{\mathcal E}
\newcommand{\calG}{\mathcal G}
\newcommand{\calH}{\mathcal H}
\newcommand{\calP}{\mathcal P}
\newcommand{\calR}{\mathcal R}
\newcommand{\calX}{\mathcal X}
\newcommand{\KL}{\operatorname{KL}}
\newcommand{\bkl}{\operatorname{kl}}
\newcommand{\TV}{\operatorname{TV}}
\newcommand{\Bin}{\operatorname{Bin}}
\newcommand{\Mult}{\operatorname{Mult}}
\newcommand{\sech}{\operatorname{sech}}
\newcommand{\arctanh}{\operatorname{arctanh}}
\newcommand{\dd}{\mathrm d}

\title{\LARGE High-Confidence Minimax Testing with Prescribed Errors}
\author{Ilmun Kim\\
Department of Mathematical Sciences \\ KAIST \\[1ex] \texttt{ilmunk@kaist.ac.kr}}
\date{\today}

\begin{document}
\maketitle

\begin{abstract}
Classical minimax lower bounds for testing are typically derived for fixed error probabilities, while high-confidence results often impose a common failure probability. We study prescribed-error testing, in which the level and the target type II error may be small and of different orders. Standard prior-based reductions generally aggregate the two errors into a single quantity and therefore do not capture their distinct roles. We develop a general lower-bound technique based on a binary reduction that preserves the separate roles of the two error targets. The reduction yields two directed Kullback--Leibler information requirements, corresponding respectively to the level and the target type II error. When both directed mixture divergences can be controlled, they combine into a binary Jeffreys divergence, leading to the logarithmic dependence on the level and the target type II error. Applying the framework to Gaussian sequence testing, multinomial uniformity testing, and continuous uniformity testing over H\"older balls, we obtain lower bounds that match corresponding high-confidence upper bounds and hence establish prescribed-error minimax rates sharp up to constant factors.
\end{abstract}

\section{Introduction}

\subsection{Prescribed-error testing}

We study minimax testing of a null class $\calP_0$ against an alternative class $\calP_1$ on a common measurable space $(\calX,\calA)$. Throughout, a possibly randomized test is a measurable function $\phi:\calX\to[0,1]$, interpreted as the conditional probability of rejecting the null hypothesis. For prescribed error targets $(\alpha,\beta)\in(0,1)^2$, we ask whether there exists a test that controls the type I and type II errors uniformly over the null and alternative classes:
\begin{align*}
  \sup_{P\in\calP_0}\int\phi\,\dd P\le \alpha,
  \qquad
  \sup_{P\in\calP_1}\int(1-\phi)\,\dd P\le \beta.
\end{align*}
Classical minimax testing theory \citep[e.g.,][]{Ingster1993,Spokoiny1996,IngsterSuslina2003} typically treats the type I and type II errors as fixed constants and studies how the required separation scales with the sample size, dimension, or smoothness of the model. Here, by contrast, $\alpha$ and $\beta$ are finite-sample design parameters that may be small and of different orders. Our objective is to determine how each target enters the minimax critical radius.

The Neyman--Pearson formulation already distinguishes the two errors: it fixes the level and maximizes power, or equivalently minimizes the type II error, subject to the level constraint \citep{NeymanPearson1933}. For example, the level may remain at a conventional value such as $\alpha=0.05$, while $\beta$ tends to zero. Conversely, one may require a vanishing level while keeping the target power fixed. Thus $\alpha$ and $\beta$ need not be comparable, suggesting that minimax separation may depend on them in qualitatively different ways.

Recent work has made the dependence on confidence explicit in several related settings. For simple-versus-simple testing with independent observations, \citet{PensiaJogLoh2024} characterize, up to universal constant factors, the sample complexity for possibly unequal type I and type II error targets. In distribution testing, \citet{DiakonikolasGouleakisPeeblesPrice2018} and \citet{DiakonikolasGouleakisKanePeeblesPrice2021} establish sharp high-confidence sample complexities for identity testing and related discrete problems, formulated in terms of a common failure probability. The classification approach of \citet{GerberHanPolyanskiy2023} gives minimax high-probability sample complexities, again under a common error probability, for discrete distributions, smooth densities, and Gaussian sequence models. In estimation, \citet{MaVerchandSamworth2026} develop high-probability analogues of classical Le Cam and Fano methods for lower-bounding minimax quantiles. Taken together, this literature highlights the role of explicit confidence parameters in nonasymptotic inference. We pursue this viewpoint in composite minimax testing. A central difficulty is that standard prior-based lower-bound techniques typically aggregate the two error probabilities and therefore do not reveal how the level and the target type II error contribute separately to minimax separation. We overcome this difficulty by introducing a binary reduction that preserves the separate roles of the level and the target type II error.

\subsection{A binary reduction for prescribed errors}

Many standard prior-based lower-bound arguments aggregate the level and the target type II error into the single quantity $\alpha+\beta$, as can be seen from the standard total-variation reduction. Let $\Pi_0$ and $\Pi_1$ be priors on the null and alternative classes, respectively, and let $Q_0$ and $Q_1$ denote the induced mixture laws. For any test $\phi$ with level at most $\alpha$ uniformly over $\calP_0$, its uniform type II error over $\calP_1$ is bounded below by
\begin{align*}
  \sup_{P\in\calP_1}\int(1-\phi)\,\dd P
  \ge
  \int(1-\phi)\,\dd Q_1
  \ge
  1-\alpha-\TV(Q_0,Q_1),
\end{align*}
where $\TV(Q_0,Q_1)=\sup_{A\in\calA}|Q_0(A)-Q_1(A)|$ is the total variation distance. Therefore, any total-variation argument proving that every level-$\alpha$ test has type II error at least $\beta$ must establish
\begin{align*}
  \TV(Q_0,Q_1)\le 1-(\alpha+\beta).
\end{align*}
The resulting condition depends on the prescribed errors only through their sum. In particular, a total-variation lower bound cannot distinguish the separate roles of the level and the target type II error. This limitation is especially apparent when $\alpha$ is fixed while $\beta\to0$: the requirement $\TV(Q_0,Q_1)\le 1-(\alpha+\beta)$ approaches the fixed bound $\TV(Q_0,Q_1)\le 1-\alpha$, and therefore does not reflect the increasingly stringent power requirement.

The usual chi-square second-moment method inherits the same limitation, since it proceeds by controlling total variation through the chi-square divergence. Classical Ingster-type minimax lower bounds \citep[e.g.,][]{Ingster1993,IngsterSuslina2003} are often analyzed through such total-variation and second-moment reductions. This aggregation phenomenon motivates our binary reduction. Rather than collapsing the two prescribed errors into the single quantity $\alpha+\beta$, the reduction yields two directional Kullback--Leibler (KL) requirements, namely, one associated with the level constraint and the other with the target type II error. These directional requirements may be used separately, or combined into a binary Jeffreys divergence that still retains the distinct roles of the two error targets.

A key advantage of this formulation is that many classical least-favourable constructions can be adapted to control the required directional divergences. We demonstrate this in three standard settings, namely, Gaussian sequence testing, multinomial uniformity testing, and continuous uniformity testing over H\"older balls. In each case, the resulting lower bounds match corresponding upper bounds and therefore yield prescribed-error minimax rates that are sharp up to constant factors.

\section{Lower bounds with prescribed type I and type II errors}
\label{sec:lower-bound}
We use the notation from the introduction. If $\Pi_0$ and $\Pi_1$ are priors on $\calP_0$ and $\calP_1$, respectively, write $Q_i=\int P\,\Pi_i(\dd P)$ for the induced mixture laws, and write $P\phi=\int\phi\,\dd P$ for test expectations. For $u,v\in(0,1)$, write
\begin{align*}
  \bkl(u\|v)
  =
  u\log\frac uv+(1-u)\log\frac{1-u}{1-v}
\end{align*}
for the binary KL divergence. We extend $\bkl$ to $[0,1]^2$ using the usual conventions $0\log 0=0$ and $c\log(c/0)=+\infty$ for $c>0$. In the nontrivial regime $\alpha+\beta<1$, define
\begin{equation}
\label{eq:Jdef}
J_{\alpha,\beta}
=
\bkl(1-\alpha\|\beta)
+
\bkl(1-\beta\|\alpha)
= (1-\alpha-\beta)\log\frac{(1-\alpha)(1-\beta)}{\alpha\beta},
\end{equation}
which is the Jeffreys divergence between $\operatorname{Bern}(1-\alpha)$ and $\operatorname{Bern}(\beta)$.

The next theorem explains why the binary Jeffreys divergence arises in prescribed-error testing. A test induces a Bernoulli decision rule, leading to two directional KL requirements below.

\begin{theorem}[Mixture lower bound for prescribed errors]
\label{thm:mixture-lower}
Let $\calP_0$ be a null class and let $\calP_1$ be a collection of alternative laws on $(\calX,\calA)$. Fix $\alpha,\beta\in(0,1)$ with $\alpha+\beta<1$. If there exists a possibly randomized test $\phi:\calX\to[0,1]$ such that
\begin{align*}
  \sup_{P\in\calP_0}P\phi\le \alpha,
  \qquad
  \sup_{P\in\calP_1}P(1-\phi)\le \beta,
\end{align*}
then, for every prior $\Pi_0$ supported on $\calP_0$ and every prior $\Pi_1$ supported on $\calP_1$, with corresponding mixtures $Q_0$ and $Q_1$,
\begin{equation}
\label{eq:onesided-main}
\KL(Q_0\|Q_1)\ge \bkl(1-\alpha\|\beta),
\qquad
\KL(Q_1\|Q_0)\ge \bkl(1-\beta\|\alpha).
\end{equation}
Consequently,
\begin{equation}
\label{eq:twosided-main}
\KL(Q_0\|Q_1)+\KL(Q_1\|Q_0)\ge J_{\alpha,\beta}.
\end{equation}
Equivalently, if some prior pair violates at least one inequality in \Cref{eq:onesided-main}, and in particular if it violates \Cref{eq:twosided-main}, then every test with level at most $\alpha$ uniformly over $\calP_0$ has type II error larger than $\beta$ at some element of $\calP_1$.
\end{theorem}

\begin{proof}[Proof of \Cref{thm:mixture-lower}]
For $i=0,1$, let $q_i=\int\phi\,\dd Q_i$. Since $\Pi_0$ and $\Pi_1$ are supported on $\calP_0$ and $\calP_1$, respectively, we have $q_0=\int(P\phi)\,\Pi_0(\dd P)\le\alpha$ and $q_1=1-\int P(1-\phi)\,\Pi_1(\dd P)\ge1-\beta$. In particular, $q_1\ge1-\beta>\alpha\ge q_0$. Let $K$ be the Markov kernel from $\calX$ to $\{0,1\}$ defined by $K(x,\{1\})=\phi(x)$ and $K(x,\{0\})=1-\phi(x)$. By the definition of $q_i$, this gives $Q_iK=\operatorname{Bern}(q_i)$. Hence, by the data-processing inequality,
\begin{align*}
  \KL(Q_0\|Q_1)\ge \bkl(q_0\|q_1)=\bkl(1-q_0\|1-q_1),
  \qquad
  \KL(Q_1\|Q_0)\ge \bkl(q_1\|q_0).
\end{align*}
For $0<v<u<1$,
\begin{align*}
  \frac{\partial}{\partial u}\bkl(u\|v)
  =
  \log\frac{u(1-v)}{v(1-u)}>0,
  \qquad
  \frac{\partial}{\partial v}\bkl(u\|v)
  =
  \frac{v-u}{v(1-v)}<0.
\end{align*}
Thus $\bkl(u\|v)$ is increasing in $u$ and decreasing in $v$ on the region $u>v$; the same monotonicity extends to the boundary with the endpoint conventions for $\bkl$. Since $1-q_0\ge1-\alpha>\beta\ge1-q_1$, we obtain
\begin{align*}
  \KL(Q_0\|Q_1)
  \ge
  \bkl(1-q_0\|1-q_1)
  \ge
  \bkl(1-\alpha\|\beta).
\end{align*}
Similarly, since $q_1\ge1-\beta>\alpha\ge q_0$, we have
\begin{align*}
  \KL(Q_1\|Q_0)
  \ge
  \bkl(q_1\|q_0)
  \ge
  \bkl(1-\beta\|\alpha).
\end{align*}
This proves \Cref{eq:onesided-main}, and summing the two inequalities gives \Cref{eq:twosided-main}. Finally, if a prior pair violates either inequality in \Cref{eq:onesided-main}, then the existence of a test satisfying both prescribed error bounds would contradict the inequalities just proved. Hence every test with level at most $\alpha$ satisfies $\sup_{P\in\calP_1}P(1-\phi)>\beta$. If \Cref{eq:twosided-main} is violated, then at least one of the two one-sided inequalities must be violated, so the same conclusion follows.
\end{proof}

We next convert \Cref{thm:mixture-lower} into a lower bound for critical radii. Let $\{\calP_1(\rho):\rho>0\}$ be a nested family of alternative classes, so that $\calP_1(\rho_2)\subseteq\calP_1(\rho_1)$ whenever $\rho_2\ge\rho_1$. Define the minimax type II risk at level $\alpha$ and separation $\rho$ by
\begin{align*}
  \calR_\alpha(\rho)
  \coloneqq
  \inf_{\phi:\,\sup_{P\in\calP_0}P\phi\le\alpha}
  \sup_{P\in\calP_1(\rho)}P(1-\phi)
\end{align*}
and the prescribed-error critical radius by
\begin{align*}
  \rho_*(\alpha,\beta)
  =
  \inf\{\rho>0:\calR_\alpha(\rho)\le\beta\}.
\end{align*}
Suppose that priors supported on $\calP_0$ and $\calP_1(\rho_0)$ induce mixtures $Q_{0,\rho_0}$ and $Q_{1,\rho_0}$. If a level-$\alpha$ test achieves type II error at most $\beta$ at separation $\rho_0$, then \Cref{thm:mixture-lower} requires the directional KL bounds \Cref{eq:onesided-main}, and hence also the Jeffreys-divergence bound \Cref{eq:twosided-main}. Therefore, any violation of these thresholds yields a lower bound on the critical radius.

\begin{corollary}[Critical-radius consequence]
\label{cor:critical-radius-consequence}
In the setup above, if, for some $\rho_0>0$, a prior pair supported on $\calP_0$ and $\calP_1(\rho_0)$ induces mixtures satisfying
\begin{align*}
  \KL(Q_{0,\rho_0}\|Q_{1,\rho_0})
  + 
  \KL(Q_{1,\rho_0}\|Q_{0,\rho_0})
  < J_{\alpha,\beta},
\end{align*}
then $\rho_*(\alpha,\beta)\ge\rho_0$. The same conclusion holds if either one-sided bound
\begin{align*}
  \KL(Q_{0,\rho_0}\|Q_{1,\rho_0})
  < \bkl(1-\alpha\|\beta)
  \quad\text{or}\quad
  \KL(Q_{1,\rho_0}\|Q_{0,\rho_0})
  < \bkl(1-\beta\|\alpha)
\end{align*}
holds.
\end{corollary}

When the total error is bounded away from one, the binary Jeffreys divergence $J_{\alpha,\beta}$ exhibits a logarithmic order in the two error targets. 

\begin{lemma}[Size of the binary Jeffreys divergence]
\label{lem:J-asymp}
For every $c_*\in(0,1)$ there is a constant $c>0$, depending only on $c_*$, such that for all $\alpha,\beta\in(0,1)$ with $\alpha+\beta\le c_*$,
\begin{align*}
  c\left(\log\frac1\alpha+\log\frac1\beta\right)
  \le
  J_{\alpha,\beta}
  \le
  \log\frac1\alpha+\log\frac1\beta.
\end{align*}
\end{lemma}
Thus, in the regime of small errors, $J_{\alpha,\beta}$ is equivalent up to constants to the sum of the two error logarithms. In particular, it is not determined by the total error $\alpha+\beta$ alone, which allows the resulting lower bounds to retain the separate roles of the two prescribed error targets. 

We next apply the lower-bound principle to several standard testing problems and show that it yields prescribed-error minimax rates sharp up to constant factors.

\section{Gaussian sequence models}
\label{sec:gaussian}

\subsection{Model and notation}

We start with the Gaussian sequence model, a canonical setting for minimax testing:
\begin{equation}
\label{eq:gaussian-model}
Y_j=\theta_j+n^{-1/2}Z_j,
\qquad
Z_j\stackrel{\mathrm{iid}}\sim N(0,1).
\end{equation}
In dimension $d$, write $P_\theta$ for the law of $(Y_1,\dots,Y_d)$ and $P_0= P_{\theta=0}$. The null is simple, $\calP_0=\{P_0\}$. The lower bounds also apply to any composite null class containing $P_{\theta=0}$, and the matching upper bounds are stated for the simple null.

We write $a\asymp_\kappa b$ to denote that $c_\kappa b\le a\le C_\kappa b$ for positive constants depending only on $\kappa$.

\subsection{Finite-dimensional dense testing}

For finite-dimensional dense alternatives, separation is measured by Euclidean distance from the origin. Let $\calG_d(\rho)=\{\theta\in\R^d:\|\theta\|_2\ge\rho\}$ and define $R^{\rm G}_{d,n,\alpha}(\rho)=\inf_{\phi:\,P_0\phi\le\alpha}\sup_{\theta\in\calG_d(\rho)}P_\theta(1-\phi)$. Let
\begin{align*}
  \rho^{\rm G}_{d,n}(\alpha,\beta)
  =
  \inf\{\rho>0:R^{\rm G}_{d,n,\alpha}(\rho)\le\beta\}
\end{align*}
be the corresponding critical radius. Two mechanisms determine the radius. The parametric branch contributes $(J_{\alpha,\beta}/n)^{1/2}$, while the dense-mixture branch contributes $(dJ_{\alpha,\beta}/n^2)^{1/4}$. The minimax radius is governed by the larger of these two scales.

\begin{theorem}[Gaussian dense testing with prescribed errors]
\label{thm:gaussian-finite-main}
Fix $c_*\in(0,1)$. For $d\ge1$, $n>0$, and $\alpha,\beta\in(0,1)$ with $\alpha+\beta\le c_*$, define
\begin{align*}
  r^{\rm G}_{d,n}(\alpha,\beta)
  =
  \max\left\{
  \left(\frac{dJ_{\alpha,\beta}}{n^2}\right)^{1/4},
  \left(\frac{J_{\alpha,\beta}}{n}\right)^{1/2}
  \right\}.
\end{align*}
Then
\begin{align*}
  \rho^{\rm G}_{d,n}(\alpha,\beta)
  \asymp_{c_*}
  r^{\rm G}_{d,n}(\alpha,\beta).
\end{align*}
\end{theorem}

When $\alpha$ and $\beta$ are fixed constants, $J_{\alpha,\beta}\asymp 1$, and the theorem recovers the classical dense-testing rate $d^{1/4}n^{-1/2}$ \citep[e.g.,][]{IngsterSuslina2003}. The proof combines parametric and dense-mixture lower bounds with a matching chi-square projection test; see \Cref{app:gaussian}.

\subsection{Sobolev ellipsoids}

For Sobolev ellipsoids, the smoothness constraint optimizes the finite-dimensional dense-mixture branch into a nonparametric branch. For the infinite version of \Cref{eq:gaussian-model}, let $\calE_\tau(M)=\{\theta\in\ell_2:\sum_{j\ge1}j^{2\tau}\theta_j^2\le M^2\}$, $\tau>0$, and $\calE_\tau(M,\rho)=\{\theta\in\calE_\tau(M):\|\theta\|_2\ge\rho\}$. The corresponding minimax type II risk is
\begin{align*}
  R^{(\tau,M)}_{n,\alpha}(\rho)
  =
  \inf_{\phi:\,P_0\phi\le\alpha}
  \sup_{\theta\in\calE_\tau(M,\rho)}P_\theta(1-\phi).
\end{align*}
We denote the corresponding prescribed-error critical radius by
\begin{align*}
  \rho^{(\tau,M)}_n(\alpha,\beta)
  =
  \inf\{\rho>0:R^{(\tau,M)}_{n,\alpha}(\rho)\le\beta\}.
\end{align*}
The term $n^{-2\tau/(4\tau+1)}J_{\alpha,\beta}^{\tau/(4\tau+1)}$ in the theorem below is the Sobolev analogue of the finite-dimensional dense-mixture branch $(dJ_{\alpha,\beta}/n^2)^{1/4}$ from the previous subsection.

\begin{theorem}[Sobolev testing with prescribed errors]
\label{thm:sobolev-main}
Fix $\tau>0$, $M>0$ and $c_*\in(0,1)$. For $\alpha,\beta\in(0,1)$ with $\alpha+\beta\le c_*$, define
\begin{align*}
  r^{(\tau,M)}_n(\alpha,\beta)
  =
  \max\left\{
  n^{-2\tau/(4\tau+1)}J_{\alpha,\beta}^{\tau/(4\tau+1)},
  \left(\frac{J_{\alpha,\beta}}{n}\right)^{1/2}
  \right\}.
\end{align*}
Then
\begin{align*}
  \rho^{(\tau,M)}_n(\alpha,\beta)
  \asymp_{\tau,M,c_*}
  \min\left\{1,r^{(\tau,M)}_n(\alpha,\beta)\right\}.
\end{align*}
\end{theorem}

When $\alpha$ and $\beta$ are fixed constants, $J_{\alpha,\beta}\asymp1$, and the theorem reduces to the classical separation rate $n^{-2\tau/(4\tau+1)}$~\citep[e.g.,][]{IngsterSuslina2003}. The lower bound uses parametric and nonparametric branch priors, whereas the upper bound truncates the quadratic statistic at the effective dimension; see \Cref{app:gaussian}.

\section{Multinomial uniformity testing}
\label{sec:multinomial}

We next consider multinomial uniformity testing. Let $u_d=(1/d,\dots,1/d)$ and observe $X=(X_1,\dots,X_d)\sim\Mult(n,p)$, equivalently, independent observations $W_1,\dots,W_n$ from $p$ on $[d]$. Write $P_p$ for the law of the full $n$-sample, or equivalently of its multinomial count vector. The null is the singleton $\calP_0=\{P_{u_d}\}$. As before, the lower bounds apply to any larger null class containing $P_{u_d}$, while the upper bound is stated for the simple uniform null. For $\delta\in(0,1]$, define
\begin{align*}
  R^{\rm unif}_{d,n,\alpha}(\delta)
  =
  \inf_{\phi:\,P_{u_d}\phi\le\alpha}
  \sup_{p:\,\|p-u_d\|_1\ge\delta}P_p(1-\phi).
\end{align*}
The corresponding critical radius is 
\begin{align*}
  \delta^{\rm unif}_{d,n}(\alpha,\beta)=\inf\{\delta > 0:R^{\rm unif}_{d,n,\alpha}(\delta)\le\beta\}.
\end{align*}
As in the Gaussian model, two mechanisms determine the radius. A Paninski sign mixture yields the dense-mixture branch $(dJ_{\alpha,\beta}/n^2)^{1/4}$, while the parametric branch contributes $(J_{\alpha,\beta}/n)^{1/2}$. The minimax radius is governed by the larger of these two scales.

\begin{theorem}[Prescribed-error radius for multinomial uniformity testing]
\label{thm:unif-main}
Fix $c_*\in(0,1)$. For every $d\ge2$, $n\in\N$, and $\alpha,\beta\in(0,1)$ with $\alpha+\beta\le c_*$, set
\begin{align*}
  r^{\rm unif}_{d,n}(\alpha,\beta)
  =
  \max\left\{
  \left(\frac{dJ_{\alpha,\beta}}{n^2}\right)^{1/4},
  \left(\frac{J_{\alpha,\beta}}{n}\right)^{1/2}
  \right\}.
\end{align*}
Then
\begin{align*}
  \delta^{\rm unif}_{d,n}(\alpha,\beta)
  \asymp_{c_*}
  \min\left\{1,r^{\rm unif}_{d,n}(\alpha,\beta)\right\}.
\end{align*}
\end{theorem}

When $\alpha$ and $\beta$ are fixed constants, $J_{\alpha,\beta}\asymp1$, and the theorem recovers the classical minimax uniformity-testing rate $d^{1/4}n^{-1/2}$~\citep[e.g.,][]{Paninski2008,DiakonikolasGouleakisPeeblesPrice2018}. The proof shows that the classical dense-mixture construction already controls both directional KL divergences required by the binary reduction; see \Cref{app:multinomial}. The matching upper bound uses the fixed-sample high-probability uniformity test recalled in \Cref{prop:known-unif-upper}.

\section{Continuous uniformity testing over H\"older balls}
\label{sec:holder}

Finally, we consider continuous uniformity testing over H\"older balls. Let $X_1,\dots,X_n$ be independent observations with density $f$ on $[0,1]^q$. Write $P_f= f^{\otimes n}$ for the joint law of the observations and $P_0= P_{f_0}$ for the uniform law. The null is the uniform density $f_0\equiv1$. For $0<s\le1$ and a bounded function $h$, write
\begin{align*}
  [h]_{C^s}
  =
  \sup_{x\ne y}
  \frac{|h(x)-h(y)|}{\|x-y\|_\infty^s},
  \qquad
  \|h\|_{C^s}=\|h\|_\infty+[h]_{C^s},
\end{align*}
which is the usual H\"older norm when $0<s<1$ and the Lipschitz norm when $s=1$. For $M>0$, $B>1$ and $\rho>0$, set
\begin{align*}
  \calH_1^s(M,B,\rho)
  =
  \left\{
   0\le f\le B,\ \int_{[0,1]^q}f=1,\
   \|f-1\|_{C^s}\le M,\
   \|f-1\|_1\ge\rho
  \right\}
\end{align*}
and define
\begin{align*}
  R^{(s,M,B)}_{n,\alpha}(\rho)
  =
  \inf_{\phi:\,P_0\phi\le\alpha}
  \sup_{f\in\calH_1^s(M,B,\rho)}P_f(1-\phi).
\end{align*}
The corresponding critical radius is
\begin{align*}
  \rho_n^{(s,M,B)}(\alpha,\beta)
  =
  \inf\{\rho>0:R^{(s,M,B)}_{n,\alpha}(\rho)\le\beta\}.
\end{align*}
The smoothness constraint restricts the class of alternatives available at separation $\rho$. Optimizing the dense-mixture branch under this constraint yields the nonparametric branch $n^{-2s/(4s+q)}J_{\alpha,\beta}^{s/(4s+q)}$, while the parametric branch contributes $(J_{\alpha,\beta}/n)^{1/2}$. The minimax radius is governed by the larger of these two scales, up to the trivial saturation of the $L_1$ distance.

\begin{theorem}[H\"older uniformity testing with prescribed errors]
\label{thm:holder-main}
Fix an integer $q\ge1$, $0<s\le1$, $M>0$, $B>1$ and $c_*\in(0,1)$. For every $n\in\N$ and $\alpha,\beta\in(0,1)$ with $\alpha+\beta\le c_*$, define
\begin{align*}
  r^{(s)}_n(\alpha,\beta)
  =
  \max\left\{
  n^{-2s/(4s+q)}J_{\alpha,\beta}^{s/(4s+q)},
  \left(\frac{J_{\alpha,\beta}}{n}\right)^{1/2}
  \right\}.
\end{align*}
Then
\begin{align*}
  \rho_n^{(s,M,B)}(\alpha,\beta)
  \asymp_{q,s,M,B,c_*}
  \min\left\{1,r^{(s)}_n(\alpha,\beta)\right\}.
\end{align*}
\end{theorem}

When $\alpha$ and $\beta$ are fixed constants, $J_{\alpha,\beta}\asymp1$, and the theorem recovers the classical minimax separation rate $n^{-2s/(4s+q)}$~\citep[e.g.,][]{Ingster1987,BalakrishnanWasserman2019}. The lower bound combines the nonparametric branch with the parametric branch, while the upper bound discretizes the observations by a regular histogram and applies a high-confidence discrete uniformity test; see \Cref{app:holder}.

\section{Discussion}
\label{sec:discussion}

We have established a binary reduction for prescribed-error minimax testing that retains the separate roles of the type I and type II error targets. The reduction yields two directed information requirements, one for each error constraint. When both directed mixture divergences can be controlled, they combine into a binary Jeffreys divergence $J_{\alpha,\beta}$, which captures the logarithmic dependence on the prescribed error parameters. This provides a simple route from many classical constant-error lower-bound constructions to corresponding prescribed-error and high-confidence results.

The examples considered here illustrate a broader phenomenon. Classical minimax lower bounds often control only one divergence direction, whereas prescribed-error lower bounds may additionally require control of the reverse direction. Whether this can be achieved depends on the structure of the testing problem and the choice of least favourable priors. In the settings studied here, the two directed mixture divergences are of the same order, leading to rates governed by a common information scale. More generally, asymmetric testing problems may exhibit different scaling in the two directions, resulting in genuinely different dependence on the type I and type II error requirements. Understanding such asymmetries remains an interesting direction for future work.

\paragraph{Acknowledgements.}
The author thanks Sivaraman Balakrishnan for pointing out the work of \citet{PensiaJogLoh2024} and for helpful discussions on minimax testing.

\bibliographystyle{plainnat}
\bibliography{references}

\appendix
\crefalias{section}{appendix}

\section[Proofs for Section 2]{Proofs for \Cref{sec:lower-bound}}
\label{app:lower-bound}

\subsection[Proof of Corollary 2.2]{Proof of \Cref{cor:critical-radius-consequence}}
\begin{proof}
Suppose first that
\begin{align*}
  \KL(Q_{0,\rho_0}\|Q_{1,\rho_0})
  +
  \KL(Q_{1,\rho_0}\|Q_{0,\rho_0})
  <
  J_{\alpha,\beta}.
\end{align*}
By continuity of $J_{\alpha,\beta}$ in $\beta$, there exists $\beta'\in(\beta,1-\alpha)$ such that
\begin{align*}
  \KL(Q_{0,\rho_0}\|Q_{1,\rho_0})
  +
  \KL(Q_{1,\rho_0}\|Q_{0,\rho_0})
  <
  J_{\alpha,\beta'}.
\end{align*}
Applying \Cref{thm:mixture-lower} with $(\alpha,\beta')$ yields $\calR_\alpha(\rho_0)\ge\beta'$. Hence $\calR_\alpha(\rho_0)>\beta$, and by nesting $\calR_\alpha(\rho)\ge\calR_\alpha(\rho_0)>\beta$ for all $\rho\le\rho_0$. Therefore $\rho_*(\alpha,\beta)\ge\rho_0$. The same argument applies when either one-sided KL inequality is violated.
\end{proof}

\subsection[Proof of binary Jeffreys divergence bound]{Proof of \Cref{lem:J-asymp}}

\begin{proof}
Let
\begin{align*}
  L=\log\frac1\alpha+\log\frac1\beta
  =\log\frac1{\alpha\beta}.
\end{align*}
Since $\alpha+\beta\le c_*<1$, the logarithm in \Cref{eq:Jdef} is positive. Moreover, $1-\alpha-\beta\le1$ and $(1-\alpha)(1-\beta)\le1$, and hence
\begin{align*}
  J_{\alpha,\beta}
  \le
  \log\frac1{\alpha\beta}
  =L.
\end{align*}
For the lower bound, set $\gamma=1-c_*>0$. Since $1-\alpha-\beta\ge \gamma$ and $(1-\alpha)(1-\beta)=1-\alpha-\beta+\alpha\beta\ge \gamma+\alpha\beta$, we have
\begin{align*}
  J_{\alpha,\beta}
  \ge
  \gamma\log\left(1+\frac{\gamma}{\alpha\beta}\right)
  =
  \gamma\log(1+\gamma e^L).
\end{align*}
Let $g(t)=\log(1+\gamma e^t)$. For every $t\ge0$,
\begin{align*}
  g'(t)=\frac{\gamma e^t}{1+\gamma e^t}\ge\frac{\gamma}{1+\gamma}.
\end{align*}
Consequently,
\begin{align*}
  g(L)\ge g(0)+\frac{\gamma}{1+\gamma}L
  \ge\frac{\gamma}{1+\gamma}L,
\end{align*}
and therefore
\begin{align*}
  J_{\alpha,\beta}
  \ge
  \frac{\gamma^2}{1+\gamma}L
  =
  \frac{(1-c_*)^2}{2-c_*}
  \left(
  \log\frac1\alpha+\log\frac1\beta
  \right).
\end{align*}
This proves the claim.
\end{proof}

\section{Proofs for the Gaussian examples}
\label{app:gaussian}

\subsection{Auxiliary Gaussian inequalities}

\begin{lemma}[Noncentral chi-square lower tail]
\label{lem:ncx2-tail}
Let $T$ have a noncentral chi-square distribution with $d$ degrees of freedom and noncentrality parameter $\lambda\ge0$. Then, for every $x>0$,
\begin{align*}
  \mP\left\{T\le d+\lambda-2\sqrt{(d+2\lambda)x}\right\}\le e^{-x}.
\end{align*}
If $T\sim\chi^2_d$, then, for every $x>0$,
\begin{align*}
  \mP\left\{T\ge d+2\sqrt{dx}+2x\right\}\le e^{-x}.
\end{align*}
\end{lemma}

\begin{proof}
The second display is equation~(4.3) of \citet[p.~1325]{LaurentMassart2000}, stated there as an immediate corollary of their Lemma~1. For the first, write $T=\sum_{j=1}^d(Z_j+\mu_j)^2$, where $\sum_j\mu_j^2=\lambda$. For every $s\ge0$, the moment generating function of $T$ yields
\begin{align*}
  \begin{aligned}
  \log\mE\exp\{-s(T-d-\lambda)\}
  &=
  s(d+\lambda)-\frac d2\log(1+2s)
  -\frac{\lambda s}{1+2s} \\
  &\le
  (d+2\lambda)s^2.
  \end{aligned}
\end{align*}
The last inequality uses $\log(1+2s)\ge2s-2s^2$ and $s\lambda-\lambda s/(1+2s)=2\lambda s^2/(1+2s)\le2\lambda s^2$. Hence, for $u\ge0$, Chernoff's bound gives
\begin{align*}
  \mP\{T-d-\lambda\le-u\}
  \le
  \inf_{s\ge0}\exp\{(d+2\lambda)s^2-su\}
  \le
  \exp\left\{-\frac{u^2}{4(d+2\lambda)}\right\}.
\end{align*}
Taking $u=2\sqrt{(d+2\lambda)x}$ proves the lower-tail bound.
\end{proof}

\begin{lemma}[Gaussian dense-mixture construction]
\label{lem:gaussian-dense-mixture}
Consider \Cref{eq:gaussian-model} in dimension $d$. For $\rho>0$, let $\Pi_\rho$ be the uniform prior on
\begin{align*}
  \Theta_d(\rho)=\left\{\theta\in\R^d:\theta_j\in\left\{\pm\frac{\rho}{\sqrt d}\right\},\ j=1,\dots,d\right\}
\end{align*}
and let $Q_\rho=\int P_\theta\,\Pi_\rho(\dd\theta)$. Then
\begin{align*}
  \KL(P_0\|Q_\rho)+\KL(Q_\rho\|P_0)
  \le
  \frac34\frac{n^2\rho^4}{d}.
\end{align*}
\end{lemma}

\begin{proof}
Set $a=\rho/\sqrt d$, $t=\sqrt n a$, and $X_j=\sqrt nY_j$. Under $P_0$, $X_1,\dots,X_d$ are independent standard normal variables. The likelihood ratio of the one-coordinate mixture $\{\pm a\}$ with respect to $N(0,n^{-1})$ is
\begin{align*}
  \exp(-t^2/2)\cosh(tX_j).
\end{align*}
Therefore, by independence across coordinates,
\begin{align*}
  L_\rho=\frac{\dd Q_\rho}{\dd P_0}
  =
  \prod_{j=1}^d\exp(-t^2/2)\cosh(tX_j).
\end{align*}
For $Z\sim N(0,1)$,
\begin{align*}
  \KL(P_0\|Q_\rho)
  =-\mE_0\log L_\rho
  =d\left\{\frac{t^2}{2}-\mE\log\cosh(tZ)\right\}.
\end{align*}
Since $|\tanh x|\le |x|$, the function $h(x)=\log\cosh x-x^2/2+x^4/12$ satisfies $h''(x)=x^2-\tanh^2x\ge0$ and $h(0)=h'(0)=0$. Hence
\begin{align*}
  \log\cosh x\ge \frac{x^2}{2}-\frac{x^4}{12},
  \qquad x\in\R.
\end{align*}
Therefore,
\begin{align*}
  \mE\log\cosh(tZ)
  \ge
  \frac{t^2}{2}-\frac{t^4}{12}\mE Z^4
  =
  \frac{t^2}{2}-\frac{t^4}{4},
\end{align*}
and hence $\KL(P_0\|Q_\rho)\le dt^4/4$. For the reverse direction, Jensen's inequality gives
\begin{align*}
  \KL(Q_\rho\|P_0)\le \log\mE_0L_\rho^2.
\end{align*}
By independence,
\begin{align*}
  \mE_0L_\rho^2
  =\left\{\mE e^{-t^2}\cosh^2(tZ)\right\}^d
  =\{\cosh(t^2)\}^d.
\end{align*}
Using $\log\cosh u\le u^2/2$,
\begin{align*}
  \KL(Q_\rho\|P_0)\le d\frac{t^4}{2}.
\end{align*}
Adding the two estimates and substituting $t^4=n^2\rho^4/d^2$ proves the claim.
\end{proof}

\subsection{Finite-dimensional dense testing}

\begin{proof}[Proof of \Cref{thm:gaussian-finite-main}] We start with the lower bound and then give the upper bound.

\medskip\noindent\textit{Lower bound.}
Write
\begin{align*}
  r_{\rm mix}=\left(\frac{dJ_{\alpha,\beta}}{n^2}\right)^{1/4},
  \qquad
  r_{\rm par}=\left(\frac{J_{\alpha,\beta}}n\right)^{1/2},
  \qquad
  r=\max\{r_{\rm mix},r_{\rm par}\}.
\end{align*}
Take $\rho=c_-r$, where $c_->0$ is a sufficiently small numerical constant. If $r=r_{\rm par}$, use the single-coordinate alternative $\theta^\star=(\rho,0,\dots,0)$. Then $\theta^\star\in\calG_d(\rho)$ and
\begin{align*}
  \KL(P_0\|P_{\theta^\star})+\KL(P_{\theta^\star}\|P_0)
  =
  n\rho^2
  =
  c_-^2J_{\alpha,\beta}
  <
  J_{\alpha,\beta}.
\end{align*}
If $r=r_{\rm mix}$, use the prior in \Cref{lem:gaussian-dense-mixture}, which is supported on $\calG_d(\rho)$. In this case,
\begin{align*}
  \KL(P_0\|Q_\rho)+\KL(Q_\rho\|P_0)
  \le
  \frac34\frac{n^2\rho^4}{d}
  =
  \frac34c_-^4J_{\alpha,\beta}
  <
  J_{\alpha,\beta}.
\end{align*}
In either branch, \Cref{cor:critical-radius-consequence} yields
\begin{align*}
  \rho^{\rm G}_{d,n}(\alpha,\beta)\ge \rho=c_-r^{\rm G}_{d,n}(\alpha,\beta),
\end{align*}
which proves the lower bound.

\medskip\noindent\textit{Upper bound.}
Let $T= n\sum_{j=1}^dY_j^2$. Under the null, $T\sim\chi^2_d$. Let $a=\log(1/\alpha)$ and reject when $T>d+2\sqrt{da}+2a$. By \Cref{lem:ncx2-tail}, this test has type I error at most $\alpha$.

Under $P_\theta$, $T$ is noncentral chi-square with noncentrality $\lambda=n\|\theta\|_2^2$. For $b=\log(1/\beta)$, \Cref{lem:ncx2-tail} shows that the type II error is at most $\beta$ whenever
\begin{align*}
  d+\lambda-2\sqrt{(d+2\lambda)b}
  \ge
  d+2\sqrt{da}+2a.
\end{align*}
The display is implied by the sufficient condition
\begin{align*}
  \lambda\ge C\{\sqrt{da}+\sqrt{db}+a+b\}
\end{align*}
with a universal constant $C$, since
\begin{align*}
  2\sqrt{(d+2\lambda)b}
  \le
  2\sqrt{db}+2\sqrt{2\lambda b}
  \le
  C'\sqrt{db}+\frac{\lambda}{2}+C'b.
\end{align*}
Since $\lambda\ge n\rho^2$ uniformly over $\calG_d(\rho)$, the stated condition gives the upper bound. Under $\alpha+\beta\le c_*<1$, \Cref{lem:J-asymp} converts the upper condition to the displayed radius scale. 
\end{proof}

\subsection{Sobolev ellipsoids}

\begin{proof}[Proof of \Cref{thm:sobolev-main}] We start with the lower bound and then give the upper bound. Write
\begin{align*}
  r_{\rm np}
  =
  n^{-2\tau/(4\tau+1)}J_{\alpha,\beta}^{\tau/(4\tau+1)},
  \qquad
  r_{\rm par}
  =
  \left(\frac{J_{\alpha,\beta}}n\right)^{1/2},
  \qquad
  r=\max\{r_{\rm np},r_{\rm par}\},
  \qquad
  \bar r=\min\{1,r\}.
\end{align*}

\medskip\noindent\textit{Lower bound.}
The parametric branch is obtained by taking $\theta^\star=(\rho,0,\dots)$, which belongs to $\calE_\tau(M,\rho)$ whenever $\rho\le M$. As in the finite-dimensional proof, the symmetrized Kullback--Leibler divergence between $P_0$ and $P_{\theta^\star}$ is $n\rho^2$.

For the nonparametric branch, choose $d=\lceil (M/\rho)^{1/\tau}/2\rceil$ and consider the hypercube supported on the first $d$ coordinates with entries $\pm\rho/\sqrt d$. For each vertex $\theta$ of this hypercube,
\begin{align*}
  \|\theta\|_2=\rho,
  \qquad
  \sum_{j\ge1}j^{2\tau}\theta_j^2
  \le
  \frac{\rho^2}{d}\sum_{j=1}^d j^{2\tau}
  \le
  \rho^2d^{2\tau}
  \le M^2.
\end{align*}
Here the last inequality uses the assumption $\rho\le M$: with $A=(M/\rho)^{1/\tau}\ge1$, one has $d=\lceil A/2\rceil\le A$. Hence the hypercube is contained in $\calE_\tau(M,\rho)$. By \Cref{lem:gaussian-dense-mixture}, the symmetrized Kullback--Leibler divergence is at most
\begin{align*}
  C\frac{n^2\rho^4}{d}
  \le
  C'n^2M^{-1/\tau}\rho^{4+1/\tau}.
\end{align*}
This is at most a fixed fraction of $J_{\alpha,\beta}$ if
\begin{align*}
  \rho\le c_{\tau,M}n^{-2\tau/(4\tau+1)}J_{\alpha,\beta}^{\tau/(4\tau+1)}.
\end{align*}
Combining the two branches and taking $\rho=c_-\bar r$, with $c_-=c_-(\tau,M,c_*)>0$ sufficiently small, gives a prior supported on $\calE_\tau(M,\rho)$ whose symmetrized Kullback--Leibler divergence from $P_0$ is strictly smaller than $J_{\alpha,\beta}$. Hence \Cref{cor:critical-radius-consequence} yields
\begin{align*}
  \rho^{(\tau,M)}_n(\alpha,\beta)\ge c_-\bar r.
\end{align*}

\medskip\noindent\textit{Upper bound.}
Write
\begin{align*}
  a=\log(1/\alpha),\qquad b=\log(1/\beta),\qquad L= a+b.
\end{align*}
By \Cref{lem:J-asymp}, under $\alpha+\beta\le c_*$ it suffices to prove the result with $L$ in place of $J_{\alpha,\beta}$, with constants depending only on $c_*$ adjusted accordingly.

Choose $C_+>M$. If $r\ge1$, then $C_+\bar r>M$, so $\calE_\tau(M,\rho)$ is empty for every $\rho\ge C_+\bar r$ and the upper bound is immediate. Hence assume $r<1$, so $\bar r=r$, and set $\rho=C_+\bar r=C_+r$. If $\rho>M$, the upper bound is again immediate, so we may assume $\rho\le M$.

Choose $d=\lceil(2M/\rho)^{1/\tau}\rceil$. For every $\theta\in\calE_\tau(M,\rho)$,
\begin{align*}
  \sum_{j>d}\theta_j^2
  \le
  d^{-2\tau}\sum_{j>d}j^{2\tau}\theta_j^2
  \le
  M^2d^{-2\tau}
  \le
  \frac{\rho^2}{4}.
\end{align*}
Hence
\begin{align*}
  \sum_{j=1}^d\theta_j^2\ge \frac{3\rho^2}{4}.
\end{align*}

Consider the truncated quadratic statistic $T_d=n\sum_{j=1}^dY_j^2$. Under $P_0$, $T_d\sim\chi_d^2$. Reject when $T_d>d+2\sqrt{da}+2a$. By \Cref{lem:ncx2-tail}, this test has type I error at most $\alpha$.

Under $P_\theta$, $T_d$ is noncentral chi-square with noncentrality
\begin{align*}
  \lambda_d= n\sum_{j=1}^d\theta_j^2
  \ge \frac34 n\rho^2.
\end{align*}
The finite-dimensional upper-bound calculation gives type II error at most $\beta$ whenever
\begin{align*}
  \lambda_d\ge C\{\sqrt{da}+\sqrt{db}+a+b\}.
\end{align*}
Since
\begin{align*}
  d\le C_\tau\left(\frac{M}{\rho}\right)^{1/\tau},
\end{align*}
it suffices that
\begin{align*}
  n\rho^2
  \ge
  C_\tau M^{1/(2\tau)}\rho^{-1/(2\tau)}L^{1/2}
  +C_\tau L.
\end{align*}
The second term follows from $\rho\ge C(L/n)^{1/2}$. The first term is implied, after increasing the constant depending on $\tau$ and $M$, by
\begin{align*}
  \rho
  \ge
  C_{\tau,M}n^{-2\tau/(4\tau+1)}
  L^{\tau/(4\tau+1)}.
\end{align*}
Thus, after enlarging $C_+=C_+(\tau,M,c_*)$, both conditions hold at $\rho=C_+\bar r$. This gives both type I and type II control.
\end{proof}

\section{Proofs for multinomial uniformity testing}
\label{app:multinomial}

\subsection{Dense-mixture construction and the binomial block bound}

For the dense-mixture construction, assume first that $d$ is even and write $K= d/2$ and $u_d=(1/d,\dots,1/d)$. The fixed-sample multinomial model is
\begin{align*}
  X=(X_1,\dots,X_d)\sim\Mult(n,p),
\end{align*}
or equivalently with independent observations $W_1,\dots,W_n$ from $p$.

\begin{lemma}[Binomial block bound]
\label{lem:binomial-block}
For $s\in\N_0$ and $\delta\in(0,3/4]$, let
\begin{align*}
  P_s=\Bin(s,1/2),
  \qquad
  Q_{s,\delta}=\frac12\Bin\left(s,\frac{1+\delta}{2}\right)+\frac12\Bin\left(s,\frac{1-\delta}{2}\right).
\end{align*}
There exists a universal constant $C>0$ such that
\begin{align*}
  \KL(P_s\|Q_{s,\delta})\le Cs(s-1)\delta^4.
\end{align*}
\end{lemma}

\begin{proof}
Let $\eta=\arctanh(\delta)$. If $X\sim P_s$ and $T_s=2X-s$, then $T_s$ has the same distribution as $\xi_1+\cdots+\xi_s$, where $\xi_1,\dots,\xi_s$ are independent Rademacher variables. Since
\begin{align*}
  1+\delta=\frac{e^\eta}{\cosh\eta},
  \qquad
  1-\delta=\frac{e^{-\eta}}{\cosh\eta},
\end{align*}
a direct calculation gives
\begin{align*}
  \frac{\dd Q_{s,\delta}}{\dd P_s}(X)
  =
  \frac{\cosh(\eta T_s)}{\cosh(\eta)^s}.
\end{align*}
Consequently,
\begin{align*}
  F_s
  \coloneqq
  \KL(P_s\|Q_{s,\delta})
  =
  s\log\cosh\eta-\mE\log\cosh(\eta T_s).
\end{align*}
Clearly, $F_0=F_1=0$. Write $g(t)=\log\cosh(\eta t)$. For $s\ge2$, conditioning on $T_{s-1}$ and using $T_s=T_{s-1}+\xi_s$ gives
\begin{align*}
  F_s-F_{s-1}
  =
  g(1)-\mE\Delta(T_{s-1}),
\end{align*}
where $\Delta(t)=\frac12\{g(t+1)+g(t-1)\}-g(t)$. Using $\cosh(a+b)\cosh(a-b)=\cosh^2a+\sinh^2b$, we obtain
\begin{align*}
  \Delta(t)
  =
  \frac12\log\left\{
  1+\sinh^2(\eta)\sech^2(\eta t)
  \right\}.
\end{align*}
Also, $g(1)=\frac12\log\{1+\sinh^2\eta\}$. Thus, setting $A=\sinh^2\eta$ and $B=\sech^2(\eta t)$, and using $\log(1+u)\le u$,
\begin{align*}
  g(1)-\Delta(t)
  &=
  \frac12\log\left(\frac{1+A}{1+AB}\right)=
  \frac12\log\left(
  1+\frac{A(1-B)}{1+AB}
  \right)\\
  &\le
  \frac12A(1-B)=
  \frac12\sinh^2(\eta)\tanh^2(\eta t).
\end{align*}
Since $\delta\le3/4$,
\begin{align*}
  \sinh^2\eta=\frac{\delta^2}{1-\delta^2}\le C\delta^2,
  \qquad
  |\eta|=\arctanh(\delta)\le C\delta.
\end{align*}
Together with $|\tanh x|\le|x|$, this yields
\begin{align*}
  g(1)-\Delta(t)
  \le
  C\delta^4t^2.
\end{align*}
Therefore,
\begin{align*}
  F_s-F_{s-1}
  \le
  C\delta^4\mE T_{s-1}^2
  =
  C\delta^4(s-1),
\end{align*}
because $T_{s-1}$ is a sum of $s-1$ independent centered Rademacher variables. Summing the increments,
\begin{align*}
  F_s
  \le
  C\delta^4\sum_{r=2}^s(r-1)
  \le
  C s(s-1)\delta^4.
\end{align*}
This proves the claim.
\end{proof}

\subsection{Lower-bound branches}

\begin{proposition}[Dense-mixture branch]
\label{prop:unif-dense-mixture-branch}
There is a universal constant $c>0$ such that, for all $\delta\in(0,1/2]$,
\begin{align*}
  \frac{n^2\delta^4}{d}\le cJ_{\alpha,\beta}
  \quad\Longrightarrow\quad
  R^{\rm unif}_{d,n,\alpha}(\delta)\ge\beta.
\end{align*}
\end{proposition}

\begin{proof}
For $\omega\in\{\pm1\}^K$, define
\begin{align*}
  p^\omega_{2k-1}=\frac{1+\delta\omega_k}{d},
  \qquad
  p^\omega_{2k}=\frac{1-\delta\omega_k}{d},
  \qquad k=1,\dots,K.
\end{align*}
Each $p^\omega$ is a probability vector and $\|p^\omega-u_d\|_1=\delta$. Let $\Pi$ be the uniform prior on $\{\pm1\}^K$ and let
\begin{align*}
  Q=\int P_{p^\omega}\Pi(\dd\omega)
\end{align*}
be the induced mixture.

By \Cref{thm:mixture-lower}, it suffices to upper-bound $\KL(P_{u_d}\|Q)+\KL(Q\|P_{u_d})$. Let
\begin{align*}
  S_k= X_{2k-1}+X_{2k},
  \qquad k=1,\dots,K.
\end{align*}
The block totals have the same law under $P_{u_d}$ and under every $P_{p^\omega}$, hence also under $Q$:
\begin{align*}
  (S_1,\dots,S_K)\sim\Mult\left(n,\frac1K,\dots,\frac1K\right),
\end{align*}
because every pair has total probability $2/d=1/K$. Conditional on $S_k=s$, under $P_{u_d}$,
\begin{align*}
  X_{2k-1}\mid S_k=s\sim\Bin(s,1/2),
\end{align*}
whereas under the mixture $Q$,
\begin{align*}
  X_{2k-1}\mid S_k=s\sim Q_{s,\delta}.
\end{align*}
Thus the block-count marginal is identical under $P_{u_d}$ and $Q$, and the chain rule for Kullback--Leibler divergence gives
\begin{align*}
  \KL(P_{u_d}\|Q)
  =
  \sum_{k=1}^K
  \mE\left[
  \KL\{P_{S_k}\|Q_{S_k,\delta}\}
  \right].
\end{align*}
By \Cref{lem:binomial-block},
\begin{align*}
  \KL(P_{u_d}\|Q)
  \le
  C\delta^4
  \sum_{k=1}^K\mE\{S_k(S_k-1)\}.
\end{align*}
Since $\mE\{S_k(S_k-1)\}=n(n-1)K^{-2}$,
\begin{align*}
  \KL(P_{u_d}\|Q)
  \le
  C\delta^4K\frac{n(n-1)}{K^2}
  \le
  C'\frac{n^2\delta^4}{d}.
\end{align*}

For the reverse direction, let $L=\dd Q/\dd P_{u_d}$. By Jensen's inequality,
\begin{align*}
  \KL(Q\|P_{u_d})\le \log\mE_{u_d}L^2.
\end{align*}
Using the sample representation $W_1,\dots,W_n$, define $b(i)= k$ if $i\in\{2k-1,2k\}$ and
\begin{align*}
  s(i)=
  \begin{cases}
  1,& i=2k-1,\\
  -1,& i=2k.
  \end{cases}
\end{align*}
Then
\begin{align*}
  \frac{\dd P_{p^\omega}}{\dd P_{u_d}}(W_1,\dots,W_n)
  =
  \prod_{\ell=1}^n
  \{1+\delta\,\omega_{b(W_\ell)}s(W_\ell)\}.
\end{align*}
Therefore
\begin{align*}
  \mE_{u_d}L^2
  =
  \mE_{\omega,\omega'}
  \left[
  1+\frac{\delta^2}{K}\sum_{k=1}^K\omega_k\omega'_k
  \right]^n.
\end{align*}
Let $\sigma_k=\omega_k\omega'_k$. Then $\sigma_1,\dots,\sigma_K$ are independent Rademacher variables. Since $1+x\le e^x$,
\begin{align*}
  \mE_{u_d}L^2
  \le
  \mE_\sigma
  \exp\left\{
  \frac{n\delta^2}{K}\sum_{k=1}^K\sigma_k
  \right\}
  =
  \left\{\cosh\left(\frac{n\delta^2}{K}\right)\right\}^K.
\end{align*}
Since $\log\cosh x\le x^2/2$,
\begin{align*}
  \KL(Q\|P_{u_d})
  \le
  K\frac{n^2\delta^4}{2K^2}
  =
  \frac{n^2\delta^4}{2K}
  =
  \frac{n^2\delta^4}{d}.
\end{align*}
Thus the symmetrized Kullback--Leibler divergence is at most $C'n^2\delta^4/d$. This proves the claim when $d$ is even.

If $d$ is odd, take $K=\lfloor d/2\rfloor$, perturb the first $2K$ coordinates exactly as above, and leave the remaining coordinate unperturbed. Rescale the perturbation amplitude by the universal factor $d/(2K)$ so that $\|p^\omega-u_d\|_1=\delta$. The rescaled amplitude is at most $3\delta/2\le3/4$, and since $K\asymp d$, the same block-count calculation gives
\begin{align*}
  \KL(P_{u_d}\|Q)+\KL(Q\|P_{u_d})
  \le
  C''\frac{n^2\delta^4}{d},
\end{align*}
after changing only the universal constant. The proposition follows from \Cref{thm:mixture-lower} after choosing $c$ small enough that the displayed symmetrized divergence is strictly smaller than $J_{\alpha,\beta}$.
\end{proof}

\begin{proposition}[Parametric branch]
\label{prop:unif-parametric-branch}
There is a universal constant $c>0$ such that, for all $\delta\in(0,1/2]$,
\begin{align*}
  n\delta^2\le cJ_{\alpha,\beta}
  \quad\Longrightarrow\quad
  R^{\rm unif}_{d,n,\alpha}(\delta)\ge\beta.
\end{align*}
\end{proposition}

\begin{proof}
Let $m=2\lfloor d/2\rfloor$ and $a=\delta d/m$. Define
\begin{align*}
  p^\star_i=\begin{cases}
  (1+a)/d,&1\le i\le m/2,\\
  (1-a)/d,&m/2<i\le m,\\
  1/d,&m<i\le d.
  \end{cases}
\end{align*}
Then $\|p^\star-u_d\|_1=\delta$. The simple null--alternative pair $(u_d,p^\star)$ therefore suffices for a lower bound. Let
\begin{align*}
  Y_\ell=
  \begin{cases}
  1,&W_\ell\le m/2,\\
  2,&m/2<W_\ell\le m,\\
  3,&m<W_\ell\le d,
  \end{cases}
  \qquad \ell=1,\dots,n.
\end{align*}
Conditional on $Y_\ell$, the locations inside each part are uniform under both laws. Hence the Kullback--Leibler divergences reduce to the corresponding three-point divergences:
\begin{align*}
  \KL(P_{u_d}\|P_{p^\star})=-\frac{nm}{2d}\log(1-a^2),
\end{align*}
and
\begin{align*}
  \KL(P_{p^\star}\|P_{u_d})=\frac{nm}{2d}\{(1+a)\log(1+a)+(1-a)\log(1-a)\}.
\end{align*}
Therefore
\begin{align*}
  \KL(P_{u_d}\|P_{p^\star})+\KL(P_{p^\star}\|P_{u_d})
  \le Cn\frac{m}{d}a^2\le C'n\delta^2,
\end{align*}
for $\delta\le1/2$, since $a\le 3/4$. The result follows from \Cref{thm:mixture-lower} after choosing $c$ small enough that the displayed symmetrized divergence is strictly smaller than $J_{\alpha,\beta}$.
\end{proof}

\subsection{High-confidence upper bound}

\begin{proposition}[Standard fixed-sample uniformity upper bound]
\label{prop:known-unif-upper}
There exists a universal constant $C>0$ such that, for every $d\ge2$, $n\in\N$, $\delta\in(0,1]$, and $\eta\in(0,1/2)$, if
\begin{align*}
  n
  \ge
  C\frac{\sqrt{d\log(1/\eta)}+\log(1/\eta)}{\delta^2},
\end{align*}
then, in the fixed-sample multinomial model, there is a test $\phi$ satisfying
\begin{align*}
  P_{u_d}(\phi=1)\le\eta,
  \qquad
  \sup_{p:\,\|p-u_d\|_1\ge\delta}P_p(\phi=0)\le\eta.
\end{align*}
\end{proposition}

\begin{proof}
This follows from the high-confidence identity-testing theorem of \citet{DiakonikolasGouleakisPeeblesPrice2018}, applied with target distribution $u_d$ and total-variation separation $\delta/2$.
\end{proof}

\subsection[Proof of multinomial uniformity theorem]{Proof of \Cref{thm:unif-main}}

\begin{proof}
Write
\begin{align*}
  r_{\rm mix}
  =
  \left(\frac{dJ_{\alpha,\beta}}{n^2}\right)^{1/4},
  \qquad
  r_{\rm par}
  =
  \left(\frac{J_{\alpha,\beta}}n\right)^{1/2},
  \qquad
  r=\max\{r_{\rm mix},r_{\rm par}\},
  \qquad
  \bar r=\min\{1,r\}.
\end{align*}

For the lower bound, take $\delta=c_-\bar r$ with $c_-\in(0,1/2]$ to be chosen, and let $A=n\delta^2$. By \Cref{prop:unif-dense-mixture-branch,prop:unif-parametric-branch}, it suffices that either
\begin{align*}
  A^2/d\le c_1J_{\alpha,\beta}
  \qquad\text{or}\qquad
  A\le c_2J_{\alpha,\beta}.
\end{align*}
Since $\bar r\le r$ and, when $r\ge1$, either $\sqrt{dJ_{\alpha,\beta}}\ge n$ or $J_{\alpha,\beta}\ge n$, choosing $c_-$ sufficiently small gives
\begin{align*}
  A\le c \Bigl\{\sqrt{dJ_{\alpha,\beta}}+J_{\alpha,\beta} \Bigr\}.
\end{align*}
If $A\le c_2J_{\alpha,\beta}$, the parametric branch applies. Otherwise $A>c_2J_{\alpha,\beta}$, so $J_{\alpha,\beta}\le A/c_2$, and hence
\begin{align*}
  A\le c\sqrt{dJ_{\alpha,\beta}}+\frac{c}{c_2}A.
\end{align*}
Taking $c<c_2/2$ gives $A\le 2c\sqrt{dJ_{\alpha,\beta}}$, whence $A^2/d\le 4c^2J_{\alpha,\beta}$. Choosing $c$ sufficiently small so that $4c^2\le c_1$, the dense-mixture branch applies. Inspecting the proofs of the branch propositions, the applicable branch provides a prior supported on alternatives at $L_1$-distance $\delta$ from $u_d$ whose symmetrized Kullback--Leibler divergence from $P_{u_d}$ is strictly smaller than $J_{\alpha,\beta}$. Therefore \Cref{cor:critical-radius-consequence} yields
\begin{align*}
  \delta^{\rm unif}_{d,n}(\alpha,\beta)\ge c_-\bar r.
\end{align*}

For the upper bound, choose $C_+>1$. If $r\ge1$, then $\delta^{\rm unif}_{d,n}(\alpha,\beta)\le1\le C_+\bar r$ by the capped definition of the critical radius. Hence assume $r<1$, so $\bar r=r$, and put $\delta=C_+\bar r=C_+r$. If $\delta>1$, the same capped bound is immediate, so assume $\delta\le1$.

Let $\eta=\alpha\wedge\beta$. Since $\alpha+\beta\le c_*<1$, one has $\eta\le(\alpha+\beta)/2\le c_*/2<1/2$, as required by \Cref{prop:known-unif-upper}. \Cref{lem:J-asymp} implies
\begin{align*}
  \log(1/\eta)
  \le
  \log(1/\alpha)+\log(1/\beta)
  \lesssim_{c_*}
  J_{\alpha,\beta}.
\end{align*}
At $\delta=C_+r$,
\begin{align*}
  n\delta^2
  =
  C_+^2nr^2
  \ge
  C_+^2\max\{\sqrt{dJ_{\alpha,\beta}},J_{\alpha,\beta}\}.
\end{align*}
Choosing $C_+=C_+(c_*)$ sufficiently large therefore implies the condition of \Cref{prop:known-unif-upper}. The resulting test has both errors at most $\eta$, and hence is level $\alpha$ with type II error at most $\beta$.
\begin{align*}
  \delta^{\rm unif}_{d,n}(\alpha,\beta)\le C_+\bar r.
\end{align*}
\end{proof}

\section{Proofs for H\"older uniformity testing}
\label{app:holder}

\subsection{Lower-bound ingredient}

The following lemma is a continuous block analogue of the symmetric two-point mixture calculation used in the multinomial lower bound. We record it separately for use on rescaled H\"older bump cells.

\begin{lemma}[Smooth block mixture]
\label{lem:smooth-block}
Let $(\calX,\mu)$ be a probability space, and let $h\in L_\infty(\mu)$ satisfy
\begin{align*}
  \int h\,\dd\mu=0,
  \qquad
  \int h^2\,\dd\mu=1.
\end{align*}
Assume that $h(U)$ and $-h(U)$ have the same distribution when $U\sim\mu$. For $0<\gamma\le(2\|h\|_\infty)^{-1}$, define
\begin{align*}
  \dd\nu_\pm=(1\pm\gamma h)\,\dd\mu.
\end{align*}
For each integer $t\ge0$, let
\begin{align*}
  P_t=\mu^{\otimes t},
  \qquad
  Q_{t,\gamma}=\frac12\nu_+^{\otimes t}+\frac12\nu_-^{\otimes t}.
\end{align*}
There is a universal constant $C>0$ such that
\begin{align*}
  \KL(P_t\|Q_{t,\gamma})+\KL(Q_{t,\gamma}\|P_t)
  \le Ct(t-1)\gamma^4.
\end{align*}
\end{lemma}

\begin{proof}
Let $H_i= h(U_i)$ for independent $U_i\sim\mu$, and write
\begin{align*}
  A_t^\pm=\prod_{i=1}^t(1\pm\gamma H_i),
  \qquad
  L_t=\frac{\dd Q_{t,\gamma}}{\dd P_t}=\frac{A_t^++A_t^-}{2}.
\end{align*}
Set $F_t=\mE[-\log L_t]$. Then $F_0=F_1=0$. Define
\begin{align*}
  D_{t-1}=\frac{A_{t-1}^+-A_{t-1}^-}{A_{t-1}^++A_{t-1}^-}.
\end{align*}
Since $L_t=L_{t-1}(1+\gamma D_{t-1}H_t)$, symmetry of $H_t$ gives, conditionally on $D_{t-1}$,
\begin{align*}
  \mE\{-\log(1+\gamma D_{t-1}H_t)\mid D_{t-1}\}
  &=-\frac12\mE\{\log(1-\gamma^2D_{t-1}^2H_t^2)\mid D_{t-1}\}\\
  &\le \gamma^2D_{t-1}^2\mE H_t^2
  =\gamma^2D_{t-1}^2.
\end{align*}
Here we used $|\gamma D_{t-1}H_t|\le1/2$ and $-\frac12\log(1-z)\le z$ for $0\le z\le1/4$. Hence
\begin{align*}
  F_t-F_{t-1}\le\gamma^2\mE D_{t-1}^2.
\end{align*}
Moreover,
\begin{align*}
  D_{t-1}
  =
  \tanh\left\{\sum_{i<t}\arctanh(\gamma H_i)\right\}.
\end{align*}
The summands are independent and centered, because $\arctanh$ is odd and $H_i$ is symmetric. Therefore, using $|\tanh z|\le|z|$ and $|\arctanh z|\le2|z|$ for $|z|\le1/2$,
\begin{align*}
  \mE D_{t-1}^2
  \le
  \mE\left\{\sum_{i<t}\arctanh(\gamma H_i)\right\}^2
  \le4(t-1)\gamma^2\mE H_1^2
  =4(t-1)\gamma^2.
\end{align*}
Summing the increments yields
\begin{align*}
  \KL(P_t\|Q_{t,\gamma})=F_t
  \le2t(t-1)\gamma^4.
\end{align*}

For the reverse direction, Jensen's inequality under $Q_{t,\gamma}$ gives
\begin{align*}
  \KL(Q_{t,\gamma}\|P_t)
  \le\log\mE_{P_t}L_t^2.
\end{align*}
Independence and the identities $\mE H_i=0$, $\mE H_i^2=1$ imply
\begin{align*}
  \mE_{P_t}L_t^2
  =
  \frac12(1+\gamma^2)^t+\frac12(1-\gamma^2)^t.
\end{align*}
With $x=\gamma^2$,
\begin{align*}
  \frac{(1+x)^t+(1-x)^t}{2}
  &=\sum_{k\ge0}\binom{t}{2k}x^{2k}\\
  &\le
  \sum_{k\ge0}\frac{\{t(t-1)\}^kx^{2k}}{(2k)!}
  =
  \cosh\Bigl\{\sqrt{t(t-1)}x\Bigr\},
\end{align*}
where $(t)_{2k}\le\{t(t-1)\}^k$. Since $\log\cosh z\le z^2/2$,
\begin{align*}
  \KL(Q_{t,\gamma}\|P_t)
  \le\frac12t(t-1)\gamma^4.
\end{align*}
Combining the two bounds proves the lemma.
\end{proof}

\subsection[Proof of H\"older uniformity theorem]{Proof of \Cref{thm:holder-main}}

\begin{proof}
Write
\begin{align*}
  r_{\rm np}
  =
  n^{-2s/(4s+q)}J_{\alpha,\beta}^{s/(4s+q)},
  \qquad
  r_{\rm par}
  =
  \left(\frac{J_{\alpha,\beta}}n\right)^{1/2},
  \qquad
  r=\max\{r_{\rm np},r_{\rm par}\},
  \qquad
  \bar r=\min\{1,r\}.
\end{align*}

\medskip\noindent
\textit{Lower bound.} For $u=(u_1,\dots,u_q)$, choose a nonzero $b\in C_c^\infty((0,1)^q)$ supported in $\{u\in(0,1)^q:u_1<1/3\}$, and set
\begin{align*}
  a(u)= c\{b(u)-b(1-u_1,u_2,\dots,u_q)\},
\end{align*}
where $c>0$ normalizes $\|a\|_1=1$. The supports of $b(u)$ and $b(1-u_1,u_2,\dots,u_q)$ are disjoint, so $a\ne0$. Moreover, with $\sigma_a^2=\int a^2\in(0,\infty)$,
\begin{align*}
  \int a=0,
  \qquad
  \int |a|=1,
  \qquad
  a(1-u_1,u_2,\dots,u_q)=-a(u_1,u_2,\dots,u_q).
\end{align*}
Thus $a(U)$ and $-a(U)$ have the same distribution for uniform $U\in[0,1]^q$.

For the parametric branch, take $f^\star=1+\rho a$. There is $\rho_0>0$, depending only on $a,M,B$, such that $f^\star\in \calH_1^s(M,B,\rho)$ whenever $0<\rho\le\rho_0$, because $\|f^\star-1\|_1=\rho$. If also $\rho\|a\|_\infty\le1/2$, then
\begin{align*}
  \KL(P_0\|P_{f^\star})+\KL(P_{f^\star}\|P_0)
  &=
  n\int \rho a\log(1+\rho a)\\
  &\le
  2n\rho^2\int a^2
  =2\sigma_a^2n\rho^2.
\end{align*}
Hence, at $\rho=c_-r_{\rm par}$, the symmetrized divergence is at most a fixed fraction of $J_{\alpha,\beta}$ when $c_->0$ is sufficiently small.

For the nonparametric branch, extend $a$ by zero to $\mathbb R^q$. Since $a$ is compactly supported in $(0,1)^q$, this extension is $C^s$. For $m\ge2$ and $k\in\{0,\dots,m-1\}^q$, define the cells
\begin{align*}
  C_{m,k}= m^{-1}(k+[0,1]^q),
  \qquad
  a_{m,k}(x)= a(mx-k),
  \qquad x\in[0,1]^q.
\end{align*}
Let $\delta_a>0$ be the distance from $\operatorname{supp}(a)$ to the boundary of $[0,1]^q$. The bump supports are disjoint and lie at distance at least $\delta_a/m$ from their cell boundaries. Consequently, there is a constant $C_a$ such that
\begin{equation}
\label{eq:holder-bump-scaling-L1}
\left\|\sum_kc_ka_{m,k}\right\|_{C^s}
\le C_am^s\max_k|c_k|.
\end{equation}
To verify the seminorm, let $F=\sum_kc_ka_{m,k}$. If $x,y$ lie in the same cell, the claim follows by scaling the $C^s$ seminorm of the zero extension of $a$. If they lie in different cells and at least one of $F(x),F(y)$ is nonzero, then $\|x-y\|_\infty\ge\delta_a/m$, so
\begin{align*}
  |F(x)-F(y)|
  \le2\|a\|_\infty\max_k|c_k|
  \le2\|a\|_\infty\delta_a^{-s}m^s\max_k|c_k|\,\|x-y\|_\infty^s.
\end{align*}
If both values vanish, then \(F(x)=F(y)=0\), and the bound is immediate. The supremum norm satisfies
\begin{align*}
  \|F\|_\infty
  \le
  \|a\|_\infty\max_k|c_k|
  \le
  \|a\|_\infty m^s\max_k|c_k|,
\end{align*}
since \(m\ge1\), and is therefore absorbed into the constant \(C_a\).

For sufficiently small $\rho$, put
\begin{align*}
  R_\rho=
  \left(\frac{M}{2C_a\rho}\right)^{1/s},
  \qquad
  m=\lfloor R_\rho\rfloor.
\end{align*}
Taking the small-radius constant so that $R_\rho\ge4$ gives $m\ge R_\rho/2\ge2$. Hence $\rho C_am^s\le M/2$ and
\begin{equation}
\label{eq:holder-number-bumps-lower-L1}
m^q\ge c_{q,s,a}\left(\frac M\rho\right)^{q/s}.
\end{equation}
For $\omega\in\{\pm1\}^{m^q}$, define
\begin{align*}
  f_\omega(x)=1+\rho\sum_k\omega_ka_{m,k}(x).
\end{align*}
By \Cref{eq:holder-bump-scaling-L1}, disjointness, and $\int a=0$, every $f_\omega$ is a density and
\begin{align*}
  \|f_\omega-1\|_{C^s}\le M.
\end{align*}
Moreover, disjointness gives
\begin{align*}
  \|f_\omega-1\|_1
  &=
  \rho\sum_k\int_{C_{m,k}}|a_{m,k}(x)|\,\dd x\\
  &=
  \rho m^qm^{-q}\int_{[0,1]^q}|a(u)|\,\dd u
  =\rho.
\end{align*}
After reducing the small-radius constant so that
\begin{align*}
  \rho\|a\|_\infty\le\min\left\{\frac12,B-1\right\},
\end{align*}
we also have $0\le f_\omega\le B$. Thus every constructed alternative \(f_\omega\) belongs to \(\calH_1^s(M,B,\rho)\).

Let $Q_m^\circ$ be the uniform mixture over $\omega$:
\begin{align*}
  Q_m^\circ
  \coloneqq
  2^{-m^q}\sum_{\omega\in\{\pm1\}^{m^q}} P_{f_\omega}.
\end{align*}
Equivalently, $Q_m^\circ$ is the law obtained by first sampling $\omega$ uniformly from $\{\pm1\}^{m^q}$ and then sampling $X_1,\dots,X_n$ independently from $f_\omega$. Every cube has mass $m^{-q}$ under the null and under every $f_\omega$. Set
\begin{align*}
  h=\frac a{\sigma_a},
  \qquad
  \gamma=\rho\sigma_a.
\end{align*}
Then $1\pm\gamma h=1\pm\rho a$, $\int h^2=1$, and $h(U)$ is symmetric. Condition on the full vector of cell labels of the $n$ observations. Its law is the same under the null and the mixture, and, conditional on it, the rescaled observations in cell $k$ have the block laws $P_{N_k}$ and $Q_{N_k,\gamma}$, independently across $k$. The posterior law of the signs remains the product uniform law because the cell probabilities do not depend on $\omega$. Here $C_{m,k}$ denotes the $k$th spatial cell, whereas $Q_{t,\gamma}$ denotes the block mixture law from \Cref{lem:smooth-block}. Thus $Q_{N_k,\gamma}$ is not a cell; it is the conditional law of the $N_k$ rescaled observations falling in cell $k$ after averaging over the random sign $\omega_k$. The KL chain rule in both directions and \Cref{lem:smooth-block} give
\begin{align*}
  \KL(P_0\|Q_m^\circ)+\KL(Q_m^\circ\|P_0)
  &\le
  C_a'\rho^4\mE_0\sum_kN_k(N_k-1)\\
  &=
  C_a'\rho^4\frac{n(n-1)}{m^q}\\
  &\le
  C_{q,s,a}'n^2M^{-q/s}\rho^{4+q/s},
\end{align*}
where the last line uses \Cref{eq:holder-number-bumps-lower-L1}. At $\rho=c_-r_{\rm np}$, the final expression is at most $C_{q,s,M,a}c_-^{4+q/s}J_{\alpha,\beta}$. Thus it is at most $J_{\alpha,\beta}/2$ for sufficiently small $c_-$.

Take now $\rho=c_-\bar r$. If $r=r_{\rm par}$, then $\rho\le c_-r_{\rm par}$, so the parametric bound applies. If $r=r_{\rm np}$, then $\rho\le c_-r_{\rm np}$, so the nonparametric mixture bound applies. Choosing $c_-$ sufficiently small ensures all preceding small-radius conditions. Hence, at this radius, we may use either the parametric construction or the nonparametric construction, according to which branch gives $\bar r$; in both cases the resulting prior is supported on $\calH_1^s(M,B,\rho)$ and has symmetrized divergence at most $J_{\alpha,\beta}/2$. \Cref{cor:critical-radius-consequence} then yields
\begin{align*}
  \rho_n^{(s,M,B)}(\alpha,\beta)\ge c_-\bar r,
\end{align*}
which proves the lower bound.

\medskip\noindent
\textit{Upper bound.} Set $\eta=\alpha\wedge\beta$ and $x=\log(1/\eta)$. By \Cref{lem:J-asymp} and the assumption $\alpha+\beta\le c_*<1$,
\begin{equation}
\label{eq:x-comparable-J-holder-L1}
x\asymp_{c_*}J_{\alpha,\beta}.
\end{equation}
Indeed, $x=\max\{\log(1/\alpha),\log(1/\beta)\}$ and $x\le\log(1/\alpha)+\log(1/\beta)\le2x$.

Set $\rho=C_+\bar r$. Since the $L_1$ distance between two densities is at most $2$, the claim is immediate if $\rho>2$. We choose $C_+>2$, so this covers the case $r\ge1$. Hence, in the remaining case,
\begin{align*}
  0<\rho\le2,
  \qquad
  \bar r=r.
\end{align*}
Let $f\in\calH_1^s(M,B,\rho)$ and write $g=f-1$. Choose
\begin{align*}
  m
  =
  \left\lceil
  \max\left\{
  2,
  \left(\frac{2M}{\rho}\right)^{1/s}
  \right\}
  \right\rceil,
  \qquad
  K=m^q.
\end{align*}
Let $\mathcal Q_m$ be the regular partition of $[0,1]^q$ into $K$ cubes, and let $\Pi_mg$ be the cell-average projection. For $x\in Q\in\mathcal Q_m$,
\begin{align*}
  |g(x)-(\Pi_mg)(x)|
  \le
  |Q|^{-1}\int_Q|g(x)-g(y)|\,\dd y
  \le
  Mm^{-s}.
\end{align*}
Consequently,
\begin{align*}
  \|g-\Pi_mg\|_1
  \le
  Mm^{-s}
  \le
  \frac\rho2.
\end{align*}
Since $\|g\|_1\ge\rho$, the triangle inequality gives
\begin{equation}
\label{eq:holder-L1-projection-signal}
\|\Pi_mg\|_1\ge\frac\rho2.
\end{equation}
For $Q\in\mathcal Q_m$, define $p_Q=\int_Qf$. Then
\begin{align*}
  \|p-u_K\|_1
  &=
  \sum_{Q\in\mathcal Q_m}
  \left|p_Q-\frac1K\right|\\
  &=
  \sum_{Q\in\mathcal Q_m}\left|\int_Qg\right|
  =
  \|\Pi_mg\|_1.
\end{align*}
By \Cref{eq:holder-L1-projection-signal}, $\|p-u_K\|_1\ge\rho/2$. Moreover,
\begin{align*}
  K=m^q
  \le
  \left(
  \left(\frac{2M}{\rho}\right)^{1/s}+1
  \right)^q
  \le
  C_{q,s,M}\rho^{-q/s},
\end{align*}
where we used $\lceil x\rceil\le x+1\le2x$ for $x\ge1$. Apply \Cref{prop:known-unif-upper} to the cell labels with $d=K$, $\delta=\rho/2$, and $\eta=\alpha\wedge\beta$. The required sample-size condition is implied by
\begin{align*}
  n
  \ge
  C_{q,s,M}
  \frac1{\rho^2}
  \left\{
  \rho^{-q/(2s)}\sqrt{x}+x
  \right\},
\end{align*}
or equivalently,
\begin{equation}
\label{eq:holder-L1-sufficient}
\rho^2
\ge
C_{q,s,M}
\left\{
\frac{\rho^{-q/(2s)}\sqrt{x}}n
+
\frac xn
\right\}.
\end{equation}
The first term in \Cref{eq:holder-L1-sufficient} is controlled whenever
\begin{align*}
  \rho
  \ge
  C_{q,s,M}
  n^{-2s/(4s+q)}x^{s/(4s+q)},
\end{align*}
and the second whenever
\begin{align*}
  \rho\ge C\left(\frac xn\right)^{1/2}.
\end{align*}
By \Cref{eq:x-comparable-J-holder-L1}, both conditions hold at $\rho=C_+\bar r=C_+r$ for a sufficiently large $C_+=C_+(q,s,M,c_*)$. The resulting test has type I error at most $\eta\le\alpha$ and, uniformly over $f\in\calH_1^s(M,B,\rho)$, type II error at most $\eta\le\beta$. Hence
\begin{align*}
  \rho_n^{(s,M,B)}(\alpha,\beta)\le C_+\bar r.
\end{align*}
Combining this with the lower bound proves the theorem.
\end{proof}

\end{document}